\documentclass[12pt,reqno]{amsart}

\usepackage{amssymb}
\usepackage{cases}
\usepackage{bbm}
\usepackage{mathabx}
\usepackage{mathrsfs}
\usepackage{graphicx}
\usepackage{amsfonts}
\usepackage{enumerate}

\usepackage{amssymb, amsmath, amsfonts, latexsym}
\usepackage{enumerate}

\newtheorem{thm}{Theorem}[section]

\newtheorem{lemma}[thm]{Lemma}
\newtheorem{proposition}[thm]{Proposition}

\newtheorem{remarks}[thm]{Remark}

\theoremstyle{definition}
\newtheorem{defn}{Definition}[section]
 \theoremstyle{remark}

\newcommand{\ee}{\mathbb{E}}

\newcommand{\rr}{\mathbb{R}}
\newcommand{\pp}{\mathbb{P}}

\newcommand{\e}{\varepsilon}
\newcommand{\GG}{\mathcal{G}}

\def\BB{\mathcal B}
\def\CC{\mathcal C}
\def\DD{\mathcal D}
\def\FF{\mathcal F}
\def\EE{\mathcal E}
\def\II{\mathcal I}

\def\PP{\mathcal P}

\def\BB{\mathcal B}
\def\CC{\mathcal C}
\def\DD{\mathcal D}
\def\FF{\mathcal F}
\def\EE{\mathcal E}
\def\HH{\mathcal H}

\def\PP{\mathcal P}

\def\<{\langle}
\def\>{\rangle}

\def\beq{\begin{equation}}
\def\nneq{\end{equation}}

\def\bdef{\begin{defn}}
\def\ndef{\end{defn}}

\def\bthm{\begin{thm}}
\def\nthm{\end{thm}}

\def\bprop{\begin{prop}}
\def\nprop{\end{prop}}

\def\brmk{\begin{remarks}}
\def\nrmk{\end{remarks}}

\def\bexa{\begin{exa}}
\def\nexa{\end{exa}}

\def\blem{\begin{lem}}
\def\nlem{\end{lem}}

\def\bcor{\begin{cor}}
\def\ncor{\end{cor}}

\def\<{\langle}
\def\>{\rangle}

\date{}

\def\bexe{\begin{exe}}
\def\nexe{\end{exe}}

\def\bprf{\begin{proof}}
\def\nprf{\end{proof}}

\def\bdes{\begin{description}}
\def\ndes{\end{description}}

\title[Moderate deviations for a stochastic wave equation in dimension three]{Moderate deviations for a stochastic wave equation in dimension three}

\author{Lingyan Cheng }
\address{Lingyan Cheng \\
Academy of Mathematics and Systems Sciences,  Chinese Academy of Sciences,  Beijing, 100190, China}
\email{chengly@amss.ac.cn}

\author{Ruinan Li}
\address{Ruinan Li \\
Academy of Mathematics and Systems Sciences,  Chinese Academy of Sciences,  Beijing, 100190, China}
\email{ruinanli@amss.ac.cn}
\author{Ran Wang}
\address{Ran Wang \\School of Mathematics and Statistics, Wuhan University,  Wuhan,  430072, China.}
\email{wangran@ustc.edu.cn}

\author{Nian Yao}
\address{Nian Yao\\College of Mathematics and Statistics, Shenzhen University, Shenzhen, 518060, China.}
\email{yaonian@szu.edu.cn}

\date{}

\begin{document}
\maketitle

 \noindent {\bf Abstract:}
 In this paper, we prove a central limit theorem and establish a moderate deviation principle for a perturbed  stochastic wave equation
 defined on $[0,T]\times \rr^3$. This equation is driven by a Gaussian noise, white in time and correlated in space.  The weak convergence approach  plays an important role.
 \vskip0.3cm

 \vskip0.3cm

\vskip0.3cm

\section{Introduction}

\noindent Since the pioneer work of Freidlin and Wentzell \cite{FW}, the theory of small perturbation large deviations for stochastic dynamics has been extensively developed,  see   books \cite{DPZ,DZ,DE}. The large deviation  principle (LDP for short) for stochastic  reaction-diffusion equations  driven by the space-time white noise was first  obtained by  Freidlin   \cite{Fre} and later by  Sowers \cite{Sowers},  Chenal and Millet \cite{CM}, Cerrai and R\"ockner  \cite{CR} and other authors.
  Also see  \cite{BDM,  RZ, XZ} and references therein for further development.
 \vskip0.3cm

 Like  large deviations, the moderate deviation  problems arise in the theory of statistical  inference quite naturally. The  moderate
 deviation  principle (MDP for short) can provide us with the rate  of convergence and a useful method for constructing asymptotic confidence intervals,
 see \cite{Erm,GZ}  and references therein.
 \vskip0.3cm

Results on the MDP for processes with independent increments were obtained in De Acosta \cite{DeA}, Ledoux \cite{Led} and so on. The study of the MDP estimates for other processes has been carried out as well, e.g., Gao \cite{Gao} for martingales,  Wu \cite{Wu} for Markov processes,  Guillin  and Liptser \cite{GL} for diffusion processes.
\vskip0.3cm
 The problem of moderate deviations for stochastic partial differential equations  has been receiving much attention in very recently years, such as Wang and Zhang \cite{WZ} for stochastic reaction-diffusion equations, Wang {\it et al.} \cite{WZZ} for stochastic Navier-Stokes equations,   Budhiraja {\it et al.} \cite{BDG} and Dong {\it et al.} \cite{DXZZ} for stochastic systems with jumps.
 Those moderate deviation  results are established for the stochastic parabolic equations.   However,  the hyperbolic case is   much more complicated, one difficulty comes from the  more complicated stochastic integral, another one comes from the lack of good regularity properties of  the Green functions. See  \cite{Dalang, DS1}  for the study of the  stochastic wave equation.
 \vskip0.3cm
Using the weak convergence approach   in \cite{BDM}, Ortiz-l\'{o}pez and Sanz-Sol\'{e} \cite{OS} proved a LDP for a stochastic wave equation defined on $[0,T]\times \rr^3$, perturbed  by a Gaussian noise which is  white in time and correlated in space.

\vskip 0.3cm
In this paper, we shall study the  central limit theorem and    moderate deviation  principle  for the stochastic wave equation in dimension $3$.

\vskip 0.3cm

The rest of this paper is organized as follows. In Section 2, we give the framework of the stochastic wave equation, and   state the main results of this paper. In Section 3, we first prove some convergence results   and then give the proof of the central limit theorem. In  Section 4, we prove the moderate deviation principle by using the weak convergence method.

\vskip0.3cm

Throughout the paper, $C(p)$ is a positive constant depending on the  parameter $p$, and $C$ is a positive constant depending on no specific parameter
(except $T$ and the Lipschitz constants), whose values may be different from line to line by convention.

\vskip0.3cm
We end this section with some notions. For any $T>0$ and $D\subset\rr^3$, let $\CC([0,T]\times D)$ be the space of all   continuous functions  from $[0,T]\times D$ to $\rr$, and let $\CC^{\alpha}([0,T]\times D)$ be the space of all H\"{o}lder continuous functions $g$ of degree $\alpha$ jointly in $(t,x)$, with the H\"older norm
$$
\|g\|_{\alpha}:=\sup_{(t,x)\neq (s,y)}\frac{|g(t,x)-g(s,y)|}{\left(|t-s|+|x-y|\right)^\alpha},
$$
and let
$$\CC^{\alpha,0}([0,T]\times D):=\left\{ g\in \CC^{\alpha}([0,T]\times D):  \lim_{\delta\rightarrow0} O_{g}(\delta)=0\right\},$$
where $O_{g}(\delta):=\sup_{|t-s|+|x-y|<\delta}\frac{|g(t,x)-g(s,y)|}{(|t-s|+|x-y|)^{\alpha}}$.  Then  $\CC^{\alpha,0}([0,T]\times D)$ is a Polish space,  which is denoted  by $\EE_\alpha$.

\section{Framework and the main results}

\subsection{Framework}

\vskip0.2cm Let us give the framework taken from Dalang and Sanz-Sol\'{e} \cite{DS1}, Ortiz-L\'{o}pez and Sanz-Sol\'{e} \cite{OS}. Consider the following  stochastic wave equation in spatial dimension  $d=3$:
\begin{equation}\label{SPDE}
    \begin{cases}
     \left(\frac{\partial^2}{\partial t^2}-\Delta\right) u^\e(t,x)=\sqrt\e\sigma\big(u^\e(t,x)\big)\dot{F}(t,x)+b\big(u^\e(t,x)\big),\\
      u^\e(0,x)=\nu_{0}(x),\\
     \frac{\partial}{\partial t}u^\e(0,x)=\tilde{\nu}_{0}(x)
    \end{cases}
\end{equation}
for all $(t,x)\in [0,T]\times\rr^3$ ($T>0$ is a fixed constant),
where $\e>0$, the coefficients $\sigma, b:\rr \rightarrow\rr $ are Lipschitz continuous functions, the term $\Delta u^{\e}$ denotes the Laplacian of $u^{\e}$ in the $x$-variable and the process $\dot{F}$ is the formal derivative of a Gaussian random field, white in time and correlated in space. Precisely, for any
$d\geq 1$, let $\DD(\rr^{d+1})$ be the space of Schwartz test functions. $F=(F(\varphi),\varphi \in \DD(\rr^{d+1}))$  is a  Gaussian process  defined on some probability space with zero mean and covariance functional
\begin{equation}\label{covariance}
E(F(\varphi)F(\psi))=\int_{\rr_+}ds\int_{\rr^d}\Gamma(dx)\left(\varphi(s)\ast\tilde{\psi}(s)\right)(x),
\end{equation}
where $\Gamma$ is a non-negative and non-negative definite tempered
measure on $\rr^d$, $\tilde{\psi}(s)(x):=\psi(s)(-x)$ and the notation $``\ast"$ means the convolution operator.
According to \cite{DM}, the  process $F$ can be extended to a   martingale measure
$$
M=\left\{M_t(A),\ t\ge0,\ A\in\BB_b(\rr^d)\right\},
$$
where $\BB_b(\rr^d)$ denotes the collection of all bounded Borel measurable sets in $\rr^d$.

Using the tempered measure $\Gamma$ above, we can define an inner product on $\DD(\rr^d)$:
$$\langle \varphi,\psi \rangle_\HH:=\int_{\rr^d}\Gamma(dx)(\varphi*\tilde{\psi})(x), \ \   \forall \varphi,\psi \in \DD(\rr^d).
$$
Let $\HH$ be the Hilbert space obtained by the completion of $\DD(\rr^d_{})$ with the inner product $\langle\cdot, \cdot\rangle_\HH $, and denote by $\|\cdot\|_{\HH}$ the induced norm.

By   Walsh's theory of stochastic integration with respect to (w.r.t. for short)  martingale measures, for any $t\geq 0$ and  $h \in \HH$, the stochastic integral
$$
B_t(h):=\int_0^t\int_{\rr^d}h(y)M(ds,dy)
$$
is well defined, and
$$
\left\{B_t^k:=\int_0^t\int_{\rr^d}e_k(y)M(ds,dy);\ k\ge1\right\}
$$
defines a sequence of independent standard Wiener processes, here $\{e_k\}_{k\ge1}$ is a complete orthonormal system of the Hilbert space $\HH$.  Thus, $B_t:=\sum_{k\ge1}B_t^k e_k$ is a cylindrical Wiener process on $\HH$. See  \cite{DPZ} or \cite{Walsh}.

\vskip0.3cm

{\bf Hypothesis (H)}:
\begin{itemize}
        \item[(H.1)] The coefficients $\sigma$ and $b$ are real Lipschitz continuous, i.e.,
       there exists some constant $K>0$ such that
\beq\label{Lip}
|\sigma(x)-\sigma(y)|\le K|x-y|, \quad|b(x)-b(y)|\le K|x-y|, \quad\forall x,y\in\rr.
\nneq

        \item[(H.2)] The spatial covariance measure $\Gamma$ is absolutely continuous with respect to the Lebesgue measure, and the density is
            $f(x)=\varphi(x)|x|^{-\beta},x\in \rr^3 \backslash \{0\}$. Here the function $\varphi$ is bounded and positive,
            $\varphi\in\CC^1(\rr^3), \nabla\varphi\in\CC^\delta_b(\rr^3)$ with  $\delta\in]0,1]$ and $\beta\in]0,2[$.
         \item[(H.3)] The initial values $\nu_0,\tilde{\nu}_0$ are bounded, $\nu_0\in \CC^2(\rr^3),\nabla \nu_0 $ is bounded,  $\triangle\nu_0$ and $\tilde{\nu}_0$ are H\"{o}lder continuous with degrees $\gamma_1,\gamma_2\in]0,1]$, respectively.
\end{itemize}
\vskip0.3cm

According to Dalang and Sanz-Sol\'e \cite{DS1}, under hypothesis ({\bf H}), Eq.\eqref{SPDE} admits a unique solution $u^\e$:
\begin{align}\label{SPDE solution}
u^{\e}(t,x)=&w(t,x)+\sqrt{\e} \sum_{k\ge 1}\int_0^t\left\langle G(t-s,x-\cdot)\sigma(u^{\e}(s,\cdot)) , e_k(\cdot)\right\rangle_{\HH}
dB_s^k\notag\\
           &+\int_0^t\big[G(t-s)\ast b(u^{\e}(s,\cdot)) \big](x)ds,
\end{align}
 where
$$
w(t,x):=\left(\frac{d}{dt}G(t)*\nu_0\right)(x)+\left(G(t)*\tilde{\nu}_0\right)(x),
$$
and $G(t)=\frac{1}{4\pi t}\sigma_t $,   $\sigma_t$ is the uniform surface measure (with total mass $4\pi t^2$) on the sphere of radius
$t$.
 Furthermore, for any $p\in[2,\infty[$,
\beq\label{eq u e estimate}
\sup_{\e\in]0,1]}\sup_{(t,x)\in[0,T]\times\rr^3}\ee\left[|u^{\e}(t,x)|^p\right]<+\infty,
\nneq
 and  for any
\beq\alpha \in
\II:= \left]0,\gamma_1\wedge\gamma_2\wedge\frac{2-\beta}{2}\wedge\frac{1+\delta}{2}\right[,
\nneq
 there exists $C>0$ such that for any $(t,x), (s,y)\in[0,T]\times D$, it holds that
$$
E\left[|u^\varepsilon(t,x)-u^\varepsilon(s,y)|^p\right]\leq C(|t-s|+|x-y|)^{\alpha p}.
 $$
  Consequently, almost all  the sample paths of the  process $\left\{u^\varepsilon(t,x);(t,x)\in[0,T]\times D\right\}$
    are $\alpha$-H\"{o}lder continuous jointly in $(t,x)$.  See Dalang and Sanz-Sol\'e \cite{DS1}  or  Hu {\it et al}. \cite{HHN} for details.

\vskip0.3cm
Intuitively, as the parameter $\e$ tends to zero, the solution $u^\e$ of $(\ref{SPDE solution})$ will tend to the solution of the deterministic  equation
\beq\label{eq u0}
u^0(t,x)=w(t,x)+\int^{t}_{0}[G(t-s)\ast b(u^0(s,\cdot))](x)ds.
\nneq
\vskip0.3cm
In this paper, we shall investigate deviations of $u^\e$ from  $u^0$, as $\e$ decreases to $0$. That is, the
asymptotic behavior of the trajectories,
\beq\label{U}
Z^\e(t,x):=\frac{1}{\sqrt{\e}h(\e)}(u^\e-u^0)(t,x),\quad(t,x)\in [0,T]\times D.
\nneq
\begin{itemize}
  \item[(LDP)]
 The case $h(\e)=1/\sqrt\e$ provides some large deviation  estimates.   Ortiz-L\'{o}pez and
 Sanz-Sol\'{e}  \cite{OS} proved that the law of the solution $u^{\e}$  satisfies a LDP, see Theorem \ref{LDP} below.
  \item[(CLT)]
  If $h(\e)$ is identically equal to $1$, we are in the domain of the central limit theorem (CLT for short).
We will show that $(u^\e-u^0)/\sqrt\e$ converges  as $\e\rightarrow 0^+$ to a  random field,   see Theorem \ref{CLT} below.
  \item[(MDP)]
  To fill in the gap between the central limit theorem scale and the large deviations scale,
we will study moderate deviations, that is when the deviation scale satisfies
\begin{equation} \label{h}
 h(\e)\to+\infty \ \ \text{and }\quad\sqrt\e h(\e)\to0,\quad \text{as}\quad\e\to0.
 \end{equation}
 In this case, we will prove that $Z^{\e}$ satisfies a LDP,  see Theorem \ref{MDP} below.  This special type of LDP is called the MDP for $u^\e$, see \cite[Section 3.7]{DZ}.

\end{itemize}

Throughout this paper, we assume \eqref{h} is in place.

\subsection{Main results}


Let $\HH_T:=L^2([0,T]; \HH)$ and consider the usual $L^2$-norm  $\|\cdot\|_{\HH_T}$ on this space. For any $h\in\HH_T$, we consider the deterministic evolution equation:
\begin{align}\label{skeleton}
V^h(t,x)&=w(t,x)+\int_0^t\left\langle G(t-s,x-\cdot)\sigma(V^h(s,\cdot)),h(s,\cdot)\right\rangle_{\HH} ds\notag\\
        &\quad +\int_0^t\left[G(t-s)\ast b(V^h(s,\cdot))\right](x)ds.
\end{align}

 By \cite[Theorem 2.3]{OS},  Eq.\eqref{skeleton} admits a unique solution   $V^h=:\mathcal{G}_1(h)\in \EE_\alpha$, where $\mathcal G_1$ is the solution functional   from $\HH_T$ to $\EE_{\alpha}$.   For any $f\in \EE_\alpha$, define
\begin{equation}\label{LDP rate function}
I_1(f)=\inf_{h\in \HH_T:\mathcal{G}_1(h)=f}\left\{\frac{1}{2}\|h\|^2_{\HH_T}\right\},
\end{equation}
with the convention $\inf\emptyset=+\infty$.

Ortiz-L\'{o}pez and Sanz-Sol\'{e} \cite{OS} proved the following LDP result for $u^\e$.
\bthm[Ortiz-L\'{o}pez and Sanz-Sol\'{e} \cite{OS}]\label{LDP}{\rm  Under assumption  {\bf(H)},   the family $\{u^\e;\ \e\in ]0,1]\}$ given by (\ref{SPDE solution}) satisfies a
large deviation  principle on $\EE_\alpha$ with the speed function $\e^{-1}$ and with the good rate function $I_1$ given by (\ref{LDP rate function}). More precisely,
\begin{itemize}
  \item[(a)]  for any $L>0$, the set $\{f\in \EE_{\alpha}; I_1(f)\le L\}$ is compact in $\EE_{\alpha}$;
  \item[(b)]for any closed subset $F\subset\EE_\alpha$,
$$ \limsup_{\e\rightarrow0^+}\e\log \pp(u^\e\in F)\leq-\inf_{f\in F}I_1(f);$$
  \item[(c)]
 for  any open subset $G\subset \EE_\alpha$,
$$\liminf_{\e\rightarrow0^+}\e\log \pp(u^\e\in G)\geq-\inf_{f\in G}I_1(f). $$
\end{itemize}
}
\nthm

\vskip0.3cm
In this paper, we further assume the condition ({\bf D}):
\begin{center}
    the function $b$ is differentiable and its derivative $b'$ is also Lipschitz.
\end{center}
More precisely, there exists a positive constant $K'$ such that
\begin{equation}\label{H3}
|b'(y)-b'(z)|\le K'|y-z|,\quad \text{for all } y,z\in\rr.\ \
\end{equation}
Combined with the Lipschitz continuity of $b$, we  conclude that
\begin{equation}\label{H3'}
|b'(z)|\le K,\quad \text{for all }  z\in\rr.
\end{equation}
\vskip0.3cm
Our first main result is the following central limit theorem.
\bthm \label{CLT}{\rm
Under conditions {\bf(H)} and {\bf (D)}, for any $\alpha\in \II$ and $p\ge 2$, the random field $(u^\e-u^0)/\sqrt{\e}$ converges in $L^p$ to a random
field $Y$ on $\EE_\alpha$, determined by
\begin{equation}\label{eq Y}
    \begin{cases}
     \left(\frac{\partial^2}{\partial t^2}-\Delta\right) Y(t,x)=\sigma(u^0(t,x))\dot{F}(t,x)+b'(u^0(t,x))Y(t,x), \\
      Y(0,x)=0,\\
     \frac{\partial}{\partial t}Y(0,x)=0,\quad t \in [0, T], \ \  x \in \rr^3.
     \end{cases}
\end{equation}
  }
\nthm
\vskip0.3cm

For any $\e>0$, let $q_{\e}:=Y/h(\e)$. Then $q_{\e}$ satisfies the following equation
\begin{equation}\label{eq qe}
          \left(\frac{\partial^2}{\partial t^2}-\Delta\right) q_{\e}(t,x)=\frac{1}{h(\e)}\sigma(u^0(t,x))\dot{F}(t,x)+b'(u^0(t,x))q_{\e}(t,x),
\end{equation}
with the same initial conditions as those of $Y$.

 Notice that   Eq.\eqref{eq qe} is a particular case of Eq.\eqref{SPDE}    if  its coefficients $\sigma$ and $b$ are allowed to depend on $(t,x)$.
Now, assume that the coefficients $\sigma$ and $b$ in Eq.\eqref{SPDE}   depend on $(t,x)$ and  they are Lipschitz continuous in the third variable uniformly over $(t,x)\in[0,T]\times\mathbb R^3$, that is, $\sigma, b:[0,T]\times\mathbb R^3\times \mathbb R\rightarrow \mathbb R$ satisfy that for all $u, v\in \mathbb R$,
$$
\sup_{t\in[0,T], x\in\mathbb R^3}\left(|\sigma(t,x,u)-\sigma(t,x,v)|+|b(t,x,u)-b(t,x,v)| \right)\le K|u-v|.
$$
By using the same strategies in   \cite{DS1} and \cite{OS},  we know that  the wave equation under above assumption admits a unique solution and the LDP result in  Theorem \ref{LDP} also holds.
 In fact, their proofs in this generalized case are the  same as that in  \cite{DS1} and \cite{OS},  only  the notions are need to be changed.   For example, see   \cite{LXZ, XZ}     for other type SPDEs.

  Hence,  $q_{\e}$ obeys a
LDP on $\EE_\alpha$ with the speed $h^2(\e)$ and with the good rate function
\begin{equation}\label{rate function1}
 I(g)=\left\{
       \begin{array}{ll}
         \inf \{\frac12\|h\|_{\HH_T}^2; Z^h=g\};   & \hbox{\text{if} $g\in \textit{Im}(Z^{\cdot})$;}\\
        +\infty, & \hbox{\text{otherwise},}
       \end{array}
     \right.
\end{equation}
where $Z^{h}$ is the solution of the following deterministic evolution equation
\begin{align}\label{eq sk}
Z^{h}(t,x)=&\int_0^t\langle G(t-s,x-\cdot)\sigma(u^0(s,\cdot)), h(s,\cdot)\rangle_{\HH} ds\notag\\
            &+\int_0^tG(t-s)\ast [b'(u^0(s,\cdot))Z^{h}(s,\cdot)](x)ds.
\end{align}

\vskip0.3cm

Our second main result is that $\{(u^\e-u^0)/[\sqrt{\e}h(\e)]\}$ satisfies the same LDP with $q_{\e}$, that is  the following theorem.
\bthm\label{MDP}{\rm Under conditions {\bf(H)} and {\bf (D)},  the family $\{(u^\e-u^0)/[\sqrt{\e}h(\e)];\ \e\in ]0,1]\}$ satisfies a
large deviation  principle on $\EE_\alpha$ with the speed function $h^2(\e) $ and with the good rate function $I$ given by \eqref{rate function1}.

 }
\nthm

 \section{Proof of Theorem \ref{CLT}}
\subsection{Convergence of solutions}

For any function $\phi:[0,T]\times\rr^3\rightarrow\rr$, let
$$|\phi|_{t,\infty} :=\sup\left\{|\phi(s,x)|:(s,x)\in[0,t]\times \rr^3\right\}.$$

The next result is concerned with the convergence of $u^\e$ as $\e\to 0$.
\begin{proposition}\label{Prop 2}{\rm
Under ${\bf(H)}$,   for any $p\ge 2$, there exists some positive constant $C(p, K, T)$ depending on $p, K, T$ such that
\begin{equation}\label{eq prop 2}
\ee\left[|u^\e-u^0|_{T, \infty}^p\right]\leq \e^{\frac{p}{2}}C(p,K,T) \rightarrow0, \quad \text{as} \ \e\rightarrow 0.
\end{equation}
}
\end{proposition}

\begin{proof} Since for any $0\le t\le T$,
\begin{align*}
 u^\e(t,x)-u^0(t,x)&=\int_0^t\left[G(t-s)\ast\left(b(u^\e(s,\cdot))-b(u^0(s,\cdot))\right)\right](x)ds\notag\\
                   &\quad +\sqrt\e\sum_{k\geq 1}\int_0^t\langle G(t-s, x-\cdot)\sigma (u^\e(s,\cdot)), e_k(\cdot)\rangle _\HH dB^k_s\notag\\
                   &=:T^\e_1(t,x)+T^\e_2(t,x),
\end{align*}
we obtain that for any $p\ge 2$,
\begin{equation}\label{T}
|u^\e-u^0|_{t,\infty}^p\leq 2^{p-1}\left(|T^\e_1|_{t,\infty}^p +|T^\e_2|_{t,\infty}^p\right).
\end{equation}
By the H\"{o}lder's inequality w.r.t the measure on $[0,T]\times \rr^3$ given by $G(t-s,dy)ds$ and the Lipschitz continuity of $b$, we obtain that
\begin{align}\label{T1}
&\ee\left[|T_1^\e|_{t,\infty}^p\right]\notag\\
 \leq & C(K) \left(\int_0^t\int_{\rr^3}G(t-s,dy)ds\right)^{p-1}\times
\int_0^t\ee\left[|u^\e-u^0|_{s,\infty}^p\right]\left(\int_{\rr^3}G(t-s,dy)\right)ds \notag\\
                     \leq & C(p, K,T)\int_0^t\ee\left[|u^\e-u^0|_{s,\infty}^p\right]ds.
\end{align}

By  Burkholder's inequality, H\"older's inequality,  the Lipschitz continuity of $\sigma$ and \eqref{eq u e estimate}, we have
\begin{align*}
& \ee\left[\left|\sum_{k\geq 1}\int_0^t\langle G(t-s,x-\cdot)\sigma(u^\e(s,\cdot)), e_k(\cdot)\rangle_\HH dB^k_s\right|^p\right]\notag \\
\leq& C(p)\ee\left[\left|\int_0^t\|G(t-s,x-\cdot)\sigma(u^\e(s,\cdot))\|^2_\HH ds\right|^{\frac{p}{2}}\right] \notag\\
\leq& C(p,K)\left(\int_0^t\int_{\rr^3}|\FF G(t-s)(\xi)|^2\mu(d\xi)ds\right)^{\frac{p}{2}-1}\notag\\
  & \times \int_0^t\left[\left( 1+ \sup_{(r,z)\in[0,s]\times\rr^3}\ee\left[|u^\e(r,z)|^p\right]\right)\int_{\rr^3}|\FF
  G(t-s)(\xi)|^2\mu(d\xi)\right] ds \notag\\
\leq& C(p,K,T)\int_0^t\left(1+\sup_{(r,z)\in[0,s]\times\rr^3}\ee[\left|u^\e(r,z)|^p\right]\right)ds <\infty,
\end{align*}
where \eqref{eq u e estimate} is used in the last inequality. The above estimate yields that
\begin{equation}\label{T2}
\ee[|T^\e_2|_{T,\infty}^p]\leq\e^{\frac{p}{2}}C(p,K,T).
\end{equation}
Putting (\ref{T})-(\ref{T2}) together, and using the Gronwall's inequality, we  obtain the desired  inequality \eqref{eq prop 2}.

The proof is complete.
\end{proof}

\vskip0.3cm

\subsection{The proof of CLT}
\noindent

The following lemma is a consequence of the Garsia-Rodemich-Rumsey's theorem, see  Millet and Sanz-Sol\'e  \cite[Lemma A2]{MS}.
\begin{lemma}\label{Lem 3}
\rm{ Let $\{V^\e(t,x);\ (t,x)\in[0,T]\times D\}$ be a family of real-valued stochastic processes. Assume that  there exists $p\in]1,\infty[$ such that
\begin{itemize}
  \item[(A1).] for any $(t,x)\in[0,T]\times D$,
     $$\lim_{\e\rightarrow 0}\ee\left[|V^{\e}(t,x)|^p\right]=0;$$
  \item[(A2).] there exists $\gamma_0>0$ such that for any $(t,x),(s,y)\in[0,T]\times D, $
$$
\ee\left[|V^\e(t,x)-V^\e(s,y)|^p\right]\leq C(|t-s|+|x-y|)^{\gamma_0+4},
$$
where $C$ is a positive constant independent of $\e$.
\end{itemize}
Then for any $\alpha\in]0,\gamma_0/p[$, $r\in[1,p[,$ we have
$$\lim_{\e\to0}\ee\left[\| V^\e\|_\alpha^r\right]=0. $$
}
\end{lemma}

\begin{proof}[Proof of Theorem \ref{CLT}] Denote $Y^\e:=(u^\e-u^0)/\sqrt{\e}$.  We will prove that for any $\alpha\in \II,\ p\geq 2$,
\begin{equation}\label{CLT p}
\lim_{\e\to0}\ee\left[\|Y^\e-Y\|^p_\alpha\right]=0.
\end{equation}
To this end, we only need to verify (A1) and (A2) in Lemma \ref{Lem 3} for $V^\e:=Y^\e-Y$.
Notice that
\begin{align}\label{I}
 &Y^\e(t,x)-Y(t,x)\notag\\
 =&\sum_{k\geq1}\int_0^t\left\langle
 G(t-s,x-\cdot)\left(\sigma(u^\e(s,\cdot))-\sigma(u^0(s,\cdot))\right),e_k(\cdot)\right\rangle_{\HH}dB^k_s\notag\\
& +\int_0^tG(t-s)\ast\left(\frac{b(u^\e(s,\cdot))-b(u^0(s,\cdot))}{\sqrt \e}-b'(u^0(s,\cdot))Y(s,\cdot)\right)(x)ds\notag\\
=:&I_1^\e(t,x)+I_2^\e(t,x)+I_3^\e(t,x),
   \end{align}
where
\begin{align*}
I_1^\e(t,x)&:=\sum_{k\geq1}\int_0^t\left\langle
G(t-s,x-\cdot)\left(\sigma(u^\e(s,\cdot))-\sigma(u^0(s,\cdot))\right),e_k(\cdot)\right\rangle_{\HH}dB^k_s,\\
I_2^\e(t,x)&:=\int_0^tG(t-s)\ast\left(\frac{b(u^\e(s,\cdot))-b(u^0(s,\cdot))}{\sqrt\e}-b'(u^0(s,\cdot))Y^\e(s,\cdot)\right)(x)ds,\\
I_3^\e(t,x)&:=\int_0^t G(t-s)\ast \left[b'(u^0(s,\cdot)\big(Y^\e(s,\cdot)-Y(s,\cdot)\big)\right](x)ds.
   \end{align*}

Next, we shall verify (A1) and (A2) for $I_i^\e, i=1,2,3$.

\noindent{\bf Step 1.}
  Following the similar calculation  in the proof of (\ref{T2}) and using the Lipschitz continuity of $\sigma$,  we  can deduce that for any $p\ge2$,
\begin{align}\label{I1}
 \ee\left[|I_1^\e|^p_{T,\infty} \right]\le C(p,K,T)\ee\left[|u^\e-u^0|_{T,\infty}^p\right] \le \e^{\frac p 2}C(p,K,T),
\end{align}
where we have used Proposition \ref{Prop 2} in the last inequality.

Notice that $u^{\e}=u^0+\sqrt\e Y^{\e}$. By the mean theorem for derivatives,  there exists a random field $v^\e(t,x)$ taking values in $(0,1)$ such that
\begin{align*}
\frac{1}{\sqrt\e}\left[b(u^\e)-b(u^0)\right]=&b'\big(u^0+\sqrt\e v^\e Y^{\e} \big)Y^\e.
\end{align*}
By the Lipschitz continuity of  $b'$, we have
\begin{align}\label{eq b1}
\frac{1}{\sqrt\e}\left[b(u^\e)-b(u^0)\right]-b'(u^0)Y^{\e}
=\left[b'\big(u^0+\sqrt\e v^\e Y^{\e} \big)-b'(u^0)\right] Y^{\e}
 \le\sqrt{\e} K' |Y^{\e}|^2.
  \end{align}
Hence
\begin{align*}
|I_2^\e(t,x)|&\le \sqrt \e K'\int_0^tG(t-s)\ast|Y^\e(s,\cdot)|^2(x)ds.
\end{align*}
By H\"{o}lder's inequality and Proposition \ref{Prop 2}, we obtain that for any $p\ge 2$,
\begin{align}\label{I2}
&\ee\left[|I_2^\e|_{t,\infty}^p\right]\notag\\
\leq&\e^{\frac p2}K'^p \left(\int_0^t\int_{\rr^3}G(t-s,dy)ds\right)^{p-1}
                          \times\int_0^t\ee\left[|Y^\e|_{s,\infty}^{2p}\right]\left(\int_{\rr^3}G(t-s,dy)\right)ds\notag\\
\leq & \e^{\frac p 2}C(p, K, K', T).
\end{align}
 By H\"{o}lder's inequality and \eqref{H3'}, we deduce that for any $p\ge 2$,
\begin{align}  \label{I3}
&\ee\left[|I_3^\e|^p_{t,\infty} \right]\notag\\
\le& {K}^p
\left(\int_0^t\int_{\rr^3}G(t-s,dy)\right)^{p-1}\times\int_0^t\ee\left[|Y^\e-Y|_{s,\infty}^{p}\right]\left(\int_{\rr^3}G(t-s,dy)\right)ds\notag\\
                             \le & C(p, K, T )\int_0^t \ee\left[|Y^\e-Y|_{s,\infty}^{p} \right] ds.
 \end{align}
Putting \eqref{I}, \eqref{I1}, \eqref{I2} and \eqref{I3} together, we have
$$
\ee\left[|Y^\e-Y|_{t,\infty} \right]^p\le C(p,K,K',T)\left(\e^{\frac{p}{2}}+\int_0^t \ee\left[|Y^\e-Y|_{s,\infty}^{p} \right] ds\right).
$$
By Gronwall's inequality, we have
\begin{equation}\label{bound V}
\ee\left[|Y^\e-Y|_{T,\infty}^p\right]\le \e^{\frac{p}{2}}C(p,K,K', T)\rightarrow 0,\quad \text{as}\ \e\rightarrow 0,
\end{equation}
which, in particular, implies  (A1) in Lemma \ref{Lem 3}.\\

\noindent{\bf Step 2.}  Notice that  $Y^\e$ satisfies that
\begin{align}\label{eq Y e}
Y^{\e}(t,x)=&\sum_{k\geq1}\int\left\langle G(t-s,x-\cdot)\sigma\left(u^0(s,\cdot)+\sqrt\e
Y^{\e}(s,\cdot)\right),e_k(\cdot)\right\rangle_{\HH}dB^k_s\notag\\
&+
\int_0^t G(t-s)\ast\frac{b\big(u^0(s,\cdot)+\sqrt{\e}Y^{\e}(s,\cdot)\big)-b(u^0(s,\cdot))}{\sqrt{\e}}(x)ds.\end{align}

For any $\e\in]0,1]$, set the mapping $\tilde \sigma_{\e}, \tilde b_\e:[0,T]\times\rr^3\times\rr\rightarrow\rr$ by
\begin{align*}
\tilde \sigma_{\e}(t,x,r):=&\sigma(u^0(t,x)+\sqrt\e r), \\
  \tilde b_\e(t,x,r):=&\frac{1}{\sqrt\e}\left[b(u^0(t,x)+\sqrt\e r)-b(u^0(t,x)) \right].
 \end{align*}
By the Lipschitz continuity of $\sigma$ and $b$, we know that $\tilde \sigma_{\e}$ and $\tilde b_{\e}$ are Lipschitz continuous in the third variable uniformly over $(t,x)\in[0,T]\times \mathbb R^3$ and $\e\in]0,1]$.
Using the same strategy of the proof for Theorem 2.3 in \cite{OS}, one can obtain   that for any $\alpha \in\II$,   $p> {4}/{\alpha}$,
\begin{equation}\label{YE}
\sup_{\e\in]0,1]}\ee\left[|Y^\e(t,x)-Y^\e(s,y)|^p\right]\leq C(|t-s|+|x-y|)^{\alpha p}
\end{equation}
and
\begin{equation}\label{Y}
\ee\left[|Y(t,x)-Y(s,y)|^p\right]\leq C(|t-s|+|x-y|)^{\alpha p}.
\end{equation}
Putting (\ref{YE}) and (\ref{Y}) together, we obtain the H\"older continuity of $V^\e$, that is for any $\alpha \in\II$,   $p> {4}/{\alpha}$, there exists a constant $C>0$ such that
\beq\label{V}
\sup_{\e\in]0,1]}\ee\left[\left|(Y^\e(t,x)-Y(t,x)\right)-\left(Y^\e(s,y)-Y(s,y))\right|^p\right]\le C(|t-s|+|x-y|)^{\alpha p}.
\nneq
By Lemma \ref{Lem 3}, \eqref{bound V} and \eqref{V}, we obtain  that for any $\tilde \alpha\in ]0, \alpha-4/p[, r\in [1,p[$,
$$
\lim_{\e\to0}\ee[\| Y^\e-Y\|^r_{\tilde{\alpha}}]=0.
$$
By the  arbitrariness of $p>  4/\alpha, r\in[1,p[$, we  get the desired result in Theorem \ref{CLT} for any $\alpha\in \II, p\ge2$.

The proof is complete.
\end{proof}

\section{Proof of Theorem \ref{MDP}}


 Let $\PP$ denote the set of predictable processes belonging to $L^2(\Omega\times[0,T];\ \HH)$. For any $N>0$, we define
$$\HH^N_T:=\{h\in \HH_T: \|h\|_{\HH_T}\leq N\},$$
 $$\PP^N_T:=\{v\in \PP: v\in\HH^N_T, a.s. \},$$
and we endow $\HH^N_T$   with the weak topology of $\HH_T.$




 Let $Z^\e:=(u^\e-u^0)/(\sqrt\e h(\e))$. Then
   \begin{align}\label{eq Z e2}
Z^{\e}(t,x)
=&\frac{1}{h(\e)}\sum_{k\ge1}\int_0^t\left\langle G(t-s,x-\cdot)\sigma\big(u^0(s,\cdot)+\sqrt\e h(\e)Z^{\e}(s,\cdot)\big),e_k(\cdot)
\right\rangle_{\HH}dB_s^k\notag\\
& +\int_0^tG(t-s)\ast\left(\frac{b\big(u^0(s,\cdot)+\sqrt{\e}h(\e)Z^{\e}(s,\cdot)\big)-b(u^0(s,\cdot))}{\sqrt\e h(\e)}\right)(x)ds.
\end{align}

\vskip0.3cm

For any $\e\in]0,1]$ and $v\in\PP_{T}^N$, consider the controlled equation $Z^{\e,v}$ defined by
\begin{align}\label{eq Z e}
&Z^{\e,v}(t,x)\notag\\
 =&\frac{1}{h(\e)}\sum_{k\ge1}\int_0^t\left\langle G(t-s,x-\cdot)\sigma\big(u^0(s,\cdot)+\sqrt\e h(\e)Z^{\e,v}(s,\cdot)\big),e_k(\cdot)
\right\rangle_{\HH}dB_s^k\notag\\
&  +\int_0^t\left\langle
G(t-s,x-\cdot)\sigma\left(u^0(s,\cdot)+\sqrt{\e}h(\e)Z^{\e,v}(s,\cdot)\right),v(s,\cdot)\right\rangle_{\HH}ds\notag\\
& +\int_0^tG(t-s)\ast\left(\frac{b\big(u^0(s,\cdot)+\sqrt{\e}h(\e)Z^{\e,v}(s,\cdot)\big)-b(u^0(s,\cdot))}{\sqrt\e h(\e)}\right)(x)ds.
\end{align}

Following  the proof of Theorem 2.3 in \cite{OS},  similarly  to \eqref{eq Y e},  one can prove that Eq.\eqref{eq Z e} admits a unique solution $\{Z^{\e,v}(t,x);\ (t,x)\in[0,T]\times\rr^3\}$ satisfying  that for any $p\in[2,\infty[$,
\beq\label{eq Z e estimate 1}
\sup_{\e\in]0,1]}\sup_{v\in\PP_{T}^N}\sup_{(t,x)\in[0,T]\times\rr^3}\ee\left[|Z^{\e,v}(t,x)|^p\right]<\infty,
\nneq
and there exists $C>0$ such that for $(t,x),(s,y)\in[0,T]\times D$ and $\alpha \in \II$,
\begin{equation}\label{Holder continuity 1}
\sup_{\e\in]0,1]}\sup_{v\in{\PP}^N_T}\ee\left[|Z^{\e,v}(t,x)-Z^{\e,v}(s,y)|^p\right]\leq C(|t-s|+|x-y|)^{\alpha p}.
\end{equation}
  Particularly, taking $v\equiv0$, we know that
for any $p\in[2,\infty[$,
\beq\label{eq Z e estimate 2}
\sup_{\e\in]0,1]} \sup_{(t,x)\in[0,T]\times\rr^3}\ee\left[|Z^{\e }(t,x)|^p\right]<\infty,
\nneq
and
\begin{equation}\label{Holder continuity 2}
\sup_{\e\in]0,1]} \ee\left[|Z^{\e }(t,x)-Z^{\e }(s,y)|^p\right]\leq C(|t-s|+|x-y|)^{\alpha p}.
\end{equation}

Recall $Z^h$ defined in Eq.\eqref{eq sk}.  Consider the following conditions:
\begin{itemize}
 \item[(a)]For any family $\{v^\e;\ \e>0\}\subset\PP_{T}^N$ which converges in distribution as $\e\rightarrow 0$ to $v\in \PP_{T}^N$, as
     $\HH_T^N$-valued random variables,
       $$
       \lim_{\e\rightarrow 0}Z^{\e,v^{\e}}=Z^v\quad \text{
       in distribution},
       $$
       as $\EE_{\alpha}$-valued random variables, where $Z^v$ denotes the solution of Eq.\eqref{eq sk} corresponding to the $\HH_T^N$-valued random variable $v$ (instead of a deterministic function $h$);
 \item[(b)]The set $\{Z^h;\ h\in \HH_T^N\}$ is a compact set of $\EE_{\alpha}$, where $Z^h$ is the solution of Eq.\eqref{eq sk}.
\end{itemize}

\bprf[The proof of Theorem \ref{MDP}]
Applying  \cite[Theorem 6]{BDM} to the solution functional  $\GG^{\e}: \CC([0,T];\rr^{\infty})\rightarrow \EE_{\alpha}: \GG^{\e}(\sqrt \e B):=Z^{\e}$, the solution of Eq.\eqref{eq Z e2}, and the solution functional  $\GG^0: \HH_T\rightarrow \EE^{\alpha}$, $\GG^0(h):=Z^h$, the solution of Eq.\eqref{eq sk}, conditions (a) and (b) above imply the MDP result in Theorem \ref{MDP}.  The verification of  condition (a)  will be given in Proposition \ref{Prop weak convergence}. Since $\HH_T^N$ is compact in the weak topology of $\HH$,  condition (b) follows from the continuity of the mapping  $ \HH_T^N \ni h\rightarrow Z^h\in\EE_{\alpha}$  which will be proved in Proposition \ref{Prop conv}.
The proof is complete.
\nprf

\vskip0.3cm
\begin{proposition}\label{Prop weak convergence}
{\rm Under conditions  {\bf(H)} and {\bf (D)},   for any  family $\{v^\e;\ \e>0\}\subset\PP_{T}^N$ which converges in distribution as $\e\rightarrow 0$ to $v\in \PP_{T}^N$, as
     $\HH_T^N$-valued random variables,  it holds that
       $$
       \lim_{\e\rightarrow 0}Z^{\e,v^{\e}}=Z^v, \quad \text{
       in distribution},
       $$
       as $\EE_{\alpha}$-valued random variables.
 }
\end{proposition}

\begin{proof}

By the Skorokhod representation theorem, there exist a probability space $(\bar\Omega,\bar\FF,(\bar\FF_t),\bar\pp)$, and, on this basis, a sequence
of independent Brownian motions $\bar B=(\bar B^k)_{k\ge1}$ and also a family of $\bar \FF_t$-predictable processes $\{\bar v^\e;\ \e>0\}$, $\bar v$  taking values on $\HH_T^N$, $\bar\pp$-a.s., such that the joint law of $(v^\e,v, B)$ under
$\pp$ coincides with that of  $(\bar v^\e,\bar v, \bar B)$ under $\bar\pp$ and
\beq\label{eq conver weak}
\lim_{\e\rightarrow0}\langle\bar v^\e-\bar v, g \rangle_{\HH_T}=0, \quad\forall g\in\HH_T,  \bar \pp\mbox{-a.s.}.  \ \
\nneq
Let $\bar Z^{\e,\bar v^\e}$ be the solution to a similar equation as \eqref{eq Z e} replacing $v$  by $\bar v^\e$ and $B$ by $\bar B$, and
let $\bar Z^{\bar v}$ be the solution of Eq.\eqref{eq sk} corresponding to the $\HH_T^N$-valued random variable $\bar v$ (instead of a deterministic function $h$).

Now, we shall prove that for any $p\ge2, \alpha\in\II$,
\beq\label{eq b conv}
\lim_{\e\rightarrow0}\bar \ee\left[ \|  \bar Z^{ \e,  \bar v^\e}-  \bar Z^{  \bar v} \|_{\alpha}^p\right]=0,
\nneq
which implies the validity of Proposition \ref{Prop weak convergence}. Here the expectation in \eqref{eq b conv} refers to the probability $\bar P$.

From now on, we drop the bars in the notation for the sake of simplicity, and we denote
$$
X^{\e,v^\e,v}:=Z^{\e,v^\e}-Z^v.
$$

By \eqref{eq Z e estimate 1} and \eqref{eq Z e estimate 2}, we know that for any $p\ge2$,
$$
\sup_{\e\in]0,1]}\sup_{v\in\HH^N_T}\sup_{(t,x)\in[0,T]\times\rr^3}\ee\left[|X^{\e,v^\e,v}(t,x)|^p\right]<\infty.
$$

According to Lemma \ref{Lem 3},  to prove \eqref{eq b conv}, it is sufficient to prove that  for any $ (t,x),(s,y)\in[0,T]\times D$ and $p\ge 2$, the following  conditions hold:
\begin{itemize}
  \item[(1)] { Pointwise convergence}:
  \beq\label{eq Z 1}
  \lim_{\e\to 0}\ee\left[|X^{\e,v^\e,v}(t,x)|^p\right] =0.
    \nneq
  \item[(2)] { Estimation of the increments}:  there exists a positive constant $C$ satisfied that
  \begin{align}\label{eq Z 2}
  \sup_{\e\in]0,1]}\ee \left[\left|X^{\e,v^\e,v} (t,x)-X^{\e,v^\e,v} (s,y)\right|^p\right]\le C(|t-s|+|x-y|)^{\alpha p}.
    \end{align}
\end{itemize}

 By \eqref{Holder continuity 1} and \eqref{Holder continuity 2}, it is easy to obtain \eqref{eq Z 2}.  Now,  it remains to prove \eqref{eq Z 1}.

 Notice that for any $(t,x)\in[0,T]\times \rr^3$,
\begin{align}\label{eq I0}
&Z^{\e,v^\e}(t,x)-Z^{v}(t,x)\notag\\
=&\frac{1}{h(\e)}\sum_{k\geq1}\int_0^t\left\langle G(t-s,x-\cdot)\sigma\big(u^0(s,\cdot)+\sqrt\e h(\e)Z^{\e,v^\e}(s,\cdot)\big), e_k(\cdot)
\right\rangle_{\HH}dB^k_s\notag\\
&+\Bigg\{\int_0^t\left\langle G(t-s,x-\cdot)\sigma\big(u^0(s,\cdot)+\sqrt\e h(\e)Z^{\e,v^\e}(s,\cdot)\big), v^\e(s,\cdot)
\right\rangle_{\HH}ds\notag\\
&\ \ \ \ \ \  -\int_0^t\left\langle G(t-s,x-\cdot)\sigma\big(u^0(s,\cdot)\big),v(s,\cdot) \right\rangle_{\HH}ds\notag\Bigg\}\\
&+\Bigg\{\int_0^tG(t-s)\ast\left[\frac{b\big(u^0(s,\cdot)+\sqrt\e h(\e)Z^{\e,v^\e}(s,\cdot)\big) -b(u^0(s,\cdot))}{\sqrt\e h(\e)}\right] (x)ds\notag\\
&\ \ \ \ \ \  -\int_0^t\left\{G(t-s)\ast\left[b'(u^0(s,\cdot))Z^v(s,\cdot)\right](x) \right\}ds\Bigg\}\notag\\
=:&A_1^\e(t,x)+A_2^\e(t,x)+A_3^\e(t,x).
\end{align}

{\bf Step 1}.  For the first term $A_1^\e(t,x)$, noticing that $u^0$ is bounded,  by   Burkholder's inequality, H\"older's inequality  and the linear growth property of $\sigma$, we have
\begin{align}\label{A1}
&\sup_{(t,x)\in[0,T]\times\rr^3}\ee\left[|A_1^{\e}(t,x)|^p\right]\notag\\
\le &\frac{1}{h^p(\e)}
\sup_{(t,x)\in[0,T]\times\rr^3}\ee\left( \int_0^t\left\|G(t-s,x-\cdot)\sigma\big(u^0(s,\cdot)+\sqrt\e h(\e)Z^{\e,v^\e}(s,\cdot)\big)\right\|_{\HH}^2ds\right)^{\frac p2}\notag\\
\leq&\frac{C(p,u^0,K)}{h^p(\e)} \left(\int_0^t\int_{\rr^3}|\FF G(t-s)(\xi)|^2\mu(d\xi)ds\right)^{\frac{p}{2}-1}\notag\\
  & \times \int_0^t\left( 1+ \sup_{(r,z)\in[0,s]\times\rr^3}\ee\left[|Z^{\e,v^\e}(r,z)|^p\right]\right)\int_{\rr^3}|\FF
  G(t-s)(\xi)|^2\mu(d\xi)ds \notag\\
\leq&\frac{C(p,u^0,K, T)}{h^p(\e)},
\end{align}
where  (\ref{eq Z e estimate 1}) is used   in the last inequality.

{\bf Step 2}.  The second term is further divided into two terms:
\begin{align}\label{A2}
A_2^{\e}(t,x)
=& \int_0^t\left\langle G(t-s,x-\cdot)\left[\sigma\big(u^0(s,\cdot)+\sqrt{\e} h(\e)Z^{\e,v^{\e}}(s,\cdot)\big)-\sigma(u^0(s,\cdot))\right],
v^{\e}(s,\cdot)\right\rangle_{\HH}ds\notag\\
&  +\int_0^t\left\langle G(t-s,x-\cdot)\sigma(u^0(s,\cdot)), v^{\e}(s,\cdot)-v(s,\cdot)\right\rangle_{\HH}ds\notag\\
=:&A_{2,1}^\e(t,x)+A_{2,2}^\e(t,x).
\end{align}

By Cauchy-Schwarz's inequality, H\"older's inequality,  the Lipschitz continuity of $\sigma$ and the fact that $v^{\e}\in\PP_T^N$, we obtain that
\begin{align}\label{A21}
&\sup_{(t,x)\in[0,T]\times\rr^3}\ee[|A_{2,1}^\e(t,x)|^p]\notag\\
 \leq & \e^{\frac p2} h^p(\e) K^p  \sup_{(t,x)\in[0,T]\times\rr^3}\ee\left[\left(\int_0^t\left\| G(t-s,x-\cdot)Z^{\e, v^{\e}}(s,\cdot)  \right\|_{\HH}^2ds\right)^{\frac p2}\cdot\left(\int_0^t \|v^{\e}(s,\cdot)\|^2_{\HH}ds\right)^{\frac p2}\right]\notag\\
 \leq & \e^{\frac p2} h^p(\e)  N^p K^p \left(\int_0^t\int_{\rr^3}|\FF G(t-s)(\xi)|^2\mu(d\xi)ds\right)^{\frac{p}{2}-1}\notag\\
  & \times \int_0^t\left(  \sup_{(r,z)\in[0,s]\times\rr^3}\ee\left[|Z^{\e,v^\e}(r,z)|^p\right]\right)\int_{\rr^3}|\FF
  G(t-s)(\xi)|^2\mu(d\xi)ds \notag\\
\leq& \e^{\frac p2} h^p(\e) C(N, p, K, T),
\end{align}
where  (\ref{eq Z e estimate 1}) is also used   in the last inequality.

Next,  we will show that
\beq\label{eq A220} \lim_{\e \rightarrow 0} \sup_{(t,x)\in[0,T]\times \rr^3}\ee\left[\left|A^{\e}_{2,2}(t,x)\right|^p\right]=0.
\nneq
  By H\"older's inequality with respect to the measure on $\rr^3$ given by $|\FF G(t-s)(\xi)|^2\mu(d\xi)$,  we obtain that for any   $(t,x)\in[0,T]\times\rr^3$,
 \begin{align*}
  &\int_0^t\|G(t-s, x-\cdot)\sigma(u^0(s,\cdot))\|_{\HH}^2ds\\
\le &C\int_0^tds\left(\int_{\rr^3} |\FF G(t-s)(\xi)|^2\mu(d\xi) \right)\times\left(1+ \sup_{(s,y)\in[0,T]\times\rr^3} |u^0(s,y)|^2 \right)\\
<&+\infty.
 \end{align*}
   This implies that for any $(t,x)\in[0,T]\times\rr^3$, the function  $\{G(t-s, x-y)\sigma(u^0(s,y));\ (s,y)\in[0,T]\times\rr^3 \}$ takes its values in $\HH_T$.   Since $v^{\e}\rightarrow v$ weakly in $\HH_T^N$,  we know that
   \beq\label{eq conv}
   \lim_{\e\rightarrow0}A^{\e}_{2,2}(t,x)=0,\  \ \text{a.s.}.
   \nneq
By Cauchy-Schwarz's inequality on the Hilbert space $\HH_T$ and the facts that $\|v^{\e}\|_{\HH_T}\le N, \|v\|_{\HH_T}\le N$, we  obtain that
\begin{align}\label{eq bound}
|A^{\e}_{2,2}(t,x)|\leq &\left(\int_0^t\|G(t-s, x-\cdot)\sigma(u^0(s,\cdot))\|_{\HH}^2
ds\right)^{\frac12}\cdot\left(\int_0^t\|v^{\e}(s,\cdot)-v(s,\cdot)\|_{\HH}^2ds\right)^{\frac12}\notag\\
\leq &C(N, T)\left(\int_0^t\|G(t-s, x-\cdot)\sigma(u^0(s,\cdot))\|_{\HH}^2ds\right)^{\frac12}\notag\\
                      \leq & C(N,K,T)<+\infty,
\end{align}
here $C(N,K,T)$ is independent of $(\e, t, x)$. By the H\"older regularity of the path-wise integral
 $$
 \int_0^t\left\langle G(t-s,x-\cdot)\sigma(u^0(s,\cdot)), v^{\e}(s,\cdot)-v(s,\cdot)\right\rangle_{\HH}ds,
 $$
(see \cite[Section 2]{OS}),  we know that a.s., $\{A^{\e}_{2,2}(t,x), (t,x)\in [0,T]\times D\}$ has H\"older continuous sample paths of degree $\alpha\in \II$ jointly in $(t,x)$, and
$$
\sup_{\e\in]0,1]}\|A^{\e}_{2,2}\|_{\alpha}<\infty,\ \  \text{a.s.}.
$$
This, in particular,  implies that $\{A^{\e}_{2,2}(t,x);\ (t,x)\in [0,T]\times D\}$ is equicontinuous.
By the arbitrariness of $D\subset\rr^3$, we known that $\{A^{\e}_{2,2}(t,x);\ (t,x)\in [0,T]\times \rr^3\}$ is equiv-continuous.  Thus, by \eqref{eq conv}, \eqref{eq bound}  and Arzel\`a-Ascoli Theorem, we know that $A^{\e}_{2,2}$  converges to $0$ in the space $\CC([0,T]\times \rr^3;\rr)$, a.s. as $\e \to 0$. This implies that
\beq\label{A22}
\lim_{\e\rightarrow 0}\sup_{(t,x)\in[0,T]\times \rr^3}|A^{\e}_{2,2}(t,x)|= 0,\quad \text{\ a.s..}
\nneq

By the dominated convergence theorem, \eqref{eq bound} and \eqref{A22}, we obtain that
  $$
   \lim_{\e\rightarrow0}\ee\left[\sup_{(t,x)\in[0,T]\times \rr^3}|A^{\e}_{2,2}(t,x)|^p\right]=0,
  $$
which is stronger than \eqref{eq A220}.

{\bf Step 3}. For the third term $A_3^{\e}$,   using the same argument in the proof of \eqref{eq b1}, we have
\begin{align*}
&|A_3^{\e}(t,x)|\notag\\
\leq &\Bigg|\int_0^t\Bigg\{G(t-s)\ast\Bigg[\frac{b\big(u^0(s,\cdot)+\sqrt{\e}h(\e)Z^{\e,v^{\e}}(s,\cdot)\big)-b(u^0(s,\cdot)
)}{\sqrt{\e}h(\e)}-b'(u^0(s,\cdot))Z^{\e,v^{\e}}(s,\cdot) \Bigg](x) \Bigg\}ds\Bigg|\notag\\
 &+\Bigg|\int_0^t\left\{G(t-s)\ast\left[b'(u^0(s,\cdot))\big(Z^{\e,v^{\e}}(s,\cdot)-Z^v(s,\cdot)\big) \right](x)
  \right\}ds\Bigg|\notag\\
  \leq &C(K')\sqrt{\e}h(\e)\int_0^t G(t-s)\ast \left|Z^{\e,v^{\e}}(s,\cdot)\right|^2(x)ds\\
  &+C(K)\int_0^t G(t-s)\ast\left|Z^{\e,v^{\e}}(s,\cdot)-Z^v(s,\cdot)\right|(x)ds.
\end{align*}

By H\"older's inequality with respect to the Lebesgue measure on $[0,t]\times\rr^3$ and \eqref{eq Z e estimate 1}, we have
\begin{align}\label{A3}
&\sup_{(t,x)\in[0,T]\times \rr^3}\ee\left[|A_3^{\e}(t,x)|^p\right]\notag\\
\leq & \e^{\frac p2} h^p(\e)C(p, K') \sup_{(t,x)\in[0,T]\times \rr^3} \left(\int_0^t\int_{\rr^3} G(t-s,x-y)dsdy\right)^{p-1}\notag\\
&\ \ \ \ \times \int_0^t\int_{\rr^3}G(t-s,x-y) \sup_{(r,z)\in[0,s]\times \rr^3}
\ee\left[\left|Z^{\e,v^{\e}}(r,z)\right|^{2p}\right]dsdy\notag\\
                         &+C(p,K) \sup_{(t,x)\in[0,T]\times \rr^3} \left(\int_0^t\int_{\rr^3} G(t-s,x-y)dsdy\right)^{p-1}\notag\\
                      &\ \ \ \ \times   \int_0^t\int_{\rr^3}G(t-s,x-y)\sup_{(r,z)\in[0,s]\times
                         \rr^3}\ee\left[\left|Z^{\e,v^{\e}}(r,z)-Z^v(r,z)\right|^p\right] dsdy\notag\\
                    \leq & \e^{\frac p2} h^p(\e)C(p, K',T) +C(p,K,T)\int_0^t\sup_{(r,z)\in[0,s]\times \rr^3}\ee\left[\left|X^{\e,v^{\e},v}(r,z)\right|^p\right] ds.
\end{align}

Putting \eqref{eq I0}, \eqref{A1}, \eqref{A21}, \eqref{A22} and \eqref{A3} together, we have
\begin{align*}
&\sup_{(s,x)\in[0,t]\times \rr^3}\ee\left[|X^{\e,v^{\e},v}(s, x)|^p\right]\notag\\
 \le&C(N,p,u^0,K,K',T)\Bigg(h^{-p}(\e)+\e^{\frac p2} h^p(\e)+\sup_{(s,x)\in[0,t ]\times \rr^3}\ee\left[|A_{2,2}^\e(s, x)|^p\right]\notag\\
 &+\int_0^t\sup_{(r,z)\in[0,s]\times \rr^3}\ee\left[\left|X^{\e,v^{\e},v}(r, z)\right|^p\right]ds \Bigg).
\end{align*}
By   Gronwall's inequality and \eqref{eq A220}, we obtain that
\begin{align*}
&\sup_{(s,x)\in[0,T]\times \rr^3}\ee\left[|X^{\e,v^{\e},v}(s,x) |^p\right]\notag\\
\le&C(N,p,u^0,K,K',T)\left(h^{-p}(\e)+\e^{\frac p2} h^p(\e)+\sup_{(s,x)\in[0,t ]\times \rr^3}\ee\left[|A_{2,2}^\e(s, x)|^p\right]\right)\notag\\
&\longrightarrow 0, \ \  \text{as } \e\rightarrow 0.
\end{align*}

The proof is complete.
\end{proof}

\vskip0.3cm

\begin{proposition}\label{Prop conv}
{\rm Under conditions {\bf(H)} and {\bf (D)},  for any $\alpha\in \II$,
 the mapping $ \HH_T^N \ni h\rightarrow Z^h\in\EE_{\alpha}$ is continuous with respect to the weak topology. }
\end{proposition}

\bprf Let  $\{h, (h_n)_{n\ge1}\}\subset \HH_T^N$ such that for any $g\in\HH_T$,
$$
\lim_{n\rightarrow\infty}\langle h_n-h, g \rangle_{\HH_T}=0.
$$
We need to prove that
\beq\label{eq convergence}
\lim_{n\rightarrow \infty}\|Z^{h_n}-Z^h\|_{\alpha}=0.
\nneq
Applying  the deterministic version of  Lemma \ref{Lem 3} to   $Z^{h_n}$ and $Z^h$, the proof of \eqref{eq convergence} can be divided into two steps :
\begin{itemize}
  \item[(1)]{  Pointwise convergence} : for any $(t,x)\in [0,T]\times D$,
  \beq\label{eq pointwise}
  \lim_{n\rightarrow\infty}\left|Z^{h_n}(t,x)-Z^h(t,x)  \right| = 0.
  \nneq
  \item[(2)]{  Estimation of the increments} : for any $(t,x), (s,y)\in [0,T]\times D$, $\alpha\in\II$,
  \begin{align}\label{eq increments}
  &\sup_{n\ge1}\left|(Z^{h_n}(t,x)-Z^h(t,x))-(Z^{h_n}(s,y)-Z^h(s,y))   \right|\notag\\
  \le& C\left(|t-s|+|x-y| \right)^{\alpha}.
  \end{align}
\end{itemize}

By using the similar (but more easier) strategy in the proof  of \cite[Theorem 2.3]{OS}, one can prove that the solution $Z^h$ of  \eqref{eq sk} satisfies that   for any $\alpha\in\II$,  there exists $C>0$ such that for any $(t,x), (s,y)\in[0,T]\times D$,
\beq\label{eq holder z}
\sup_{h\in \HH_T^N }|Z^h(t,x)-Z^h(s,y)|\leq C(|t-s|+|x-y|)^{\alpha }.
 \nneq
Thus, \eqref{eq increments} holds. Next, it remains to prove \eqref{eq pointwise}.

  Notice that for any $(t,x)\in [0,T]\times \rr^3$,
  \begin{align}\label{eq Point 1}
&Z^{h_n}(t,x)-Z^h(t,x)\notag\\
=&\int_0^t\left\langle G(t-s, x-\cdot)\sigma(u^0(s,\cdot)),h_n(s,\cdot)-h(s,\cdot)\right\rangle_{\HH}ds\notag\\
&+\int_0^tG(t-s)\ast\left\{b'(u^0(s,\cdot))\left[Z^{h_n}(s,\cdot)-Z^h(s,\cdot)\right]\right\}(x)ds\notag\\
=:&I_1^n(t,x)+I_2^n(t,x).
\end{align}
 Using the similar   arguments as in the
proof  \eqref{A22}, we can obtain that
\beq\label{eq I 1}
\lim_{n\rightarrow \infty}\sup_{(t,x)\in[0,T]\times\rr^3}|I_1^n(t,x)|=0.
\nneq

Set $\zeta^n(t):=\sup_{(s,x)\in[0,t]\times\rr^3}|Z^{h_n}(s,x)-Z^h(s,x)|$. By  \eqref{H3'}, we have
\begin{align}\label{eq I 2}
|I_2(t,x)|
\le &\int_0^t\int_{\rr^3}G(t-s, x-y)\left| b'(u^0(s,y))\left[Z^{h_n}(s,y)-Z^h(s,y)\right]\right|ds dy\notag\\
\le & K \int_0^t\int_{\rr^3}G(t-s, x-y)\sup_{(u,z)\in[0,s]\times\rr^3}\left| Z^{h_n}(u,z)-Z^h(u,z)\right|ds dy\notag\\
\le& C(K,T)\int_0^t \zeta^n(s)ds.
\end{align}
By \eqref{eq Point 1} and \eqref{eq I 2}, we have
$$
\zeta^n(t)\le C(K,T)\int_0^t \zeta^n(s)ds+\sup_{(t,x)\in[0,T]\times\rr^3}|I_1^n(t,x)|.
$$
Hence, by   Gronwall's lemma and \eqref{eq I 1}, we obtain that
$$ \zeta^n(T)\le e^{C(K, T)T} \sup_{(t,x)\in[0,T]\times\rr^3}|I_1^n(t,x)|\longrightarrow 0,\quad \text{as } n\rightarrow\infty.
$$
The proof is complete.
\nprf

\vskip0.3cm

\noindent{\bf Acknowledgements:}  The authors are grateful to the referee for his/her very numerous and conscientious
comments and corrections.   R. Wang was supported by Natural Science Foundation of China (11301498, 11431014, 11671076). N. Yao was supported by Natural Science Foundation of China (11371283).

\end{document}